\documentclass{article}
\usepackage[utf8]{inputenc}
\usepackage[top = 1in, bottom = 1in, left =1in, right = 1in]{geometry}
\usepackage{amsfonts,amscd,amssymb,amsmath,amsthm, mathrsfs}
\usepackage{tikz-cd}
\usepackage{xparse}
\usepackage{enumerate}
\usepackage{multirow}
\usepackage{amsthm}
\usepackage{thmtools}
\usepackage{fullpage}
\usepackage{url}

\usepackage{marginnote}
\usepackage{verbatim}
\usepackage{mathtools}
\usepackage{enumerate}

\usetikzlibrary{arrows,matrix,positioning,fit}
\usetikzlibrary{pgfplots.groupplots}

\usepackage[all]{xy}
\xyoption{arc}

\usepackage{ctable,xparse}

\usepackage{wrapfig,subcaption}
\usepackage[colorlinks=true,citecolor=blue]{hyperref}

\graphicspath{ {images/} }

\usepackage[
    backend=bibtex,
    style=alphabetic,
    sortlocale=deDE, 
    natbib=true,
    url=false, 
    doi=true,
    eprint=false,
    isbn=false,
    maxbibnames=10
]{biblatex}
\newcommand{\R}{\mathbb{R}}

\newcommand{\F}{\mathscr{F}}
\newcommand{\G}{\mathcal{G}}
\newcommand{\Q}{\mathbb{Q}}

\newcommand{\Z}{\mathbb{Z}}
\newcommand{\T}{\mathscr{T}}

\newcommand{\LP}{\operatorname{L}}

\declaretheorem{theorem}
\declaretheorem[sibling=theorem]{lemma}

\declaretheorem[sibling=theorem]{proposition}
\declaretheorem[sibling=theorem]{corollary}

\newcommand{\Image}{\operatorname{Im}}
\renewcommand{\ker}{\operatorname{Ker}}
\newcommand{\Id}{\operatorname{Id}}

\newcommand{\tr}{\operatorname{tr}}

\newcommand{\E}{\mathbb{E}}
\newcommand{\prob}{\mathbb{P}}

\renewcommand{\Pr}{\prob}
\DeclareDocumentCommand \one { o }
{%
\IfNoValueTF {#1}
{\mathbf{1}  }
{\mathbf{1}\{{#1}\} }%
}

\newcommand{\rank}{\operatorname{rank}}

\newcommand{\Bernoulli}{\operatorname{Bernoulli}}
\newcommand{\Binomial}{\operatorname{Binom}}

\newcommand{\lawequals}{\overset{\mathscr{L}}{=}}

\DeclareDocumentCommand{\Prto} {o} {
\IfNoValueTF {#1}
 {\overset{\Pr}{\longrightarrow}}
 { \xrightarrow[ #1 \to \infty]{\Pr }}
}
\DeclareDocumentCommand{\Asto} {o} {
\IfNoValueTF {#1}
 {\overset{\operatorname{a.s.}}{\longrightarrow}}
 {
 \xrightarrow[ #1 \to \infty]{\operatorname{a.s.} }
 }
}
\DeclareDocumentCommand{\Mgfto} {o} {
\IfNoValueTF {#1}
{\overset{\operatorname{mgf}}{\longrightarrow}}
{ \xrightarrow[ #1 \to \infty]{\operatorname{mgf} }}
}

\DeclareDocumentCommand{\Wkto} {o} {
\IfNoValueTF {#1}
 {\overset{(d)}{\longrightarrow}}
 { \xrightarrow[ #1 \to \infty]{(d) }}
}

\DeclareDocumentCommand \LPto { O{1} }
{\overset{\operatorname{\LP^{#1}}}{\longrightarrow}} 

\title{Simplex links in determinantal hypertrees} 

\author{Andrew Vander Werf \thanks{The author gratefully acknowledges partial support from NSF-DMS \# 1547357.}
}

\date{\today}

\begin{document}

\maketitle

\begin{abstract}
We deduce a structurally inductive description of the determinantal probability measure associated with Kalai's celebrated enumeration result for higher--dimensional spanning trees of the $n-1$--simplex. As a consequence, we derive the marginal distributions of the simplex links in such random trees. Along the way, we also characterize the higher--dimensional spanning trees of every other simplicial cone in terms of the higher--dimensional rooted forests of the underlying simplicial complex. We also apply these new results to random topology, the spectral analysis of random graphs, and the theory of high dimensional expanders. One particularly interesting corollary of these results is that the fundamental group of a union of $o(\log n)$ determinantal 2--trees has Kazhdan's property (T) with high probability. 
\end{abstract}

\section{Introduction}
 
Forty years ago, Kalai \cite{Kalai} introduced, to spectacular effect, a generalization of the graph--theoretic notion of a tree to higher--dimensional simplicial complexes, called \emph{$\Q$--acyclic simplicial complexes} for the triviality of their rational reduced homology groups in every dimension. However, recent authors \cite{LP}, \cite{KN}, \cite{Meszaros} appear to have settled on simply calling these \emph{hypertrees}. For $0\leq k<n$, let $\mathscr{T}_{n,k}$ denote the set of $k$--dimensional hypertrees on the vertex set $[n]:=\{1,2,\dots n\}$. Kalai noticed, among other things, that $\tilde{H}_{k-1}(T)$ (assume integer coefficients throughout) is a finite group for all $T\in\T_{n,k}$ and moreover that $$\sum_{T\in\mathscr{T}_{n,k}}|\tilde{H}_{k-1}(T)|^2=n^{{n-2}\choose k},$$
which is seen to be a generalization of Caley's formula by recalling that $|\tilde H_0(T)|=1$ for all $T\in\mathscr{T}_{n,1}$, due to trees being connected. This suggests a natural probability measure \cite{Lyons}, \cite{Lyons2} $\nu=\nu_{n,k}$, on $\mathscr{T}_{n,k}$ defined on atoms by $\nu_{n,k}(T)=n^{-{{n-2}\choose k}}|\tilde H_{k-1}(T)|^2$. 

Seemingly unrelated to this measure, consider the $k$--dimensional \emph{Linial--Meshulam complex}, denoted $\mathcal{Y}_k(n,p)$ and defined \cite{HomCon}, \cite{MeshWall} to be the random $k$--dimensional simplicial complex on $[n]$ with full $k-1$--skeleton wherein each $k$--face is included independently and with probability $p$. Let $\mu_{n,k}$ denote the probability density for $\mathcal{Y}_k(n,(n+1)^{-1})$. 

Let $\mathcal{T}_{n,k}$ denote a random complex distributed according to $\nu_{n,k}$, and let $\mathcal{Y}_{n,k}$ denote a random complex distributed according to $\mu_{n,k}$. Our main result is the following structure theorem for random hypertrees distributed according to $\nu_{n,k}$ which, for the special case of $k$--dimensional spanning trees of the $n-1$--simplex, answers in the affirmative a question posed by Lyons (\cite{Lyons}, Question 10.1) concerning the existence of natural disjoint--union couplings of certain determinantal measures:  \begin{theorem}\label{main}
Assume that $1\leq k<n-1$. There exists a coupling of $\mathcal{T}_{n,k}$, $\mathcal{T}_{n-1,k}$, $\mathcal{T}_{n-1,k-1}$, $\mathcal{Y}_{n-1,k}$, $\mathcal{Y}_{n-1,k-1}$ such that $\mathcal{T}_{n-1,k}$ and $\mathcal{T}_{n-1,k-1}$ are independent of $\mathcal{Y}_{n-1,k}$ and $\mathcal{Y}_{n-1,k-1}$ respectively, $\mathcal{T}_{n-1,k}$ and $\mathcal{T}_{n-1,k-1}$ are conditionally independent given $\mathcal{T}_{n,k}$, and $$\mathcal{T}_{n,k}=\operatorname{Cone}(n,\mathcal{T}_{n-1,k-1}\cup\mathcal{Y}_{n-1,k-1})\bigcup\mathcal{T}_{n-1,k}\setminus\mathcal{Y}_{n-1,k}.$$ 
\end{theorem}
We can easily identify this coned term as being the link of the vertex $n$ in $\mathcal{T}_{n,k}$. So, with this coupling, we have
$\operatorname{Link}(n,\mathcal{T}_{n,k})\lawequals\mathcal{T}_{n-1,k-1}\cup\mathcal{Y}_{n-1,k-1}$. This can be taken to mean that a vertex link in $\mathcal{T}_{n,k}$
can be simulated by first sampling $\mathcal{T}_{n-1,k-1}$ and then adding each missing $k-1$--face independently with probability $1/n$. The remaining term in the displayed union gives a description of those $k$--faces in $\mathcal{T}_{n,k}$
which do not contain the vertex $n$. This set of faces can be simulated by sampling $\mathcal{T}_{n-1,k}$ and then deleting each of its $k$--faces independently with probability $1/n$. 
Note then that these two binomial processes, one of adding $k-1$--faces which are then coned with $n$ and one of deleting $k$--faces, must be correlated at least to the point of producing the same number of faces---this is because all hypertrees of a given dimension and vertex count have the same number of top--dimensional faces (\cite{Kalai}, Proposition 2).

This simple idea of decomposing a hypertree into two collections of faces---those which contain a designated vertex and those which do not---turns out to be quite powerful. However, the idea is nothing new. For example, we can see this decomposition in use by Linial and Peled \cite{LP} 
at Kalai's suggestion 
to inductively construct collapsible hypertrees. It was noted there that if $Y\in\mathscr{T}_{n-1,k-1}$, $X\in\mathscr{T}_{n-1,k}$, and both are collapsible, then $X\cup\operatorname{Cone}(n,Y)\in\mathscr{T}_{n,k}$ and is collapsible. Corollary \ref{split} shows that this construction produces a hypertree whether or not $X$ and $Y$ are collapsible. 

By iterating our formula for the law of a vertex link, we can also determine a similar expression for the law of the link of a simplex of arbitrary dimension. 
\begin{theorem}\label{Link}
The link of the simplex $\{n-j+1,\dots, n\}$ in $\mathcal{T}_{n,k}$ is equivalent in law to 
$\mathcal{T}_{n-j,k-j}\cup\mathcal{Y}_{k-j}(n-j,j/n)$
where these two $k-j$--dimensional complexes are independent. 
\end{theorem}

\begin{proof}
We prove this by induction. The base case follows as discussed from Theorem \ref{main}. Inducting, we have 
\begin{align*}
\operatorname{Link}(\{n-j,\dots,n\},\mathcal{T}_{n,k})&\lawequals\operatorname{Link}(n-j,\mathcal{T}_{n-j,k-j}\cup\mathcal{Y}_{k-j}(n-j,j/n)) \\ 
&=\operatorname{Link}(n-j,\mathcal{T}_{n-j,k-j})\cup\operatorname{Link}(n-j,\mathcal{Y}_{k-j}(n-j,j/n)) \\ 
&\lawequals\mathcal{T}_{n-j-1,k-j-1}\cup\mathcal{Y}_{k-j-1}(n-j-1,1/(n-j))\cup\mathcal{Y}_{k-j-1}(n-j-1,j/n)) \\ 
&\lawequals\mathcal{T}_{n-j-1,k-j-1}\cup\mathcal{Y}_{k-j-1}\left(n-j-1,\frac{1}{n-j}+\frac{j}{n}-\frac{j}{n(n-j)}\right) \\ 
&=\mathcal{T}_{n-j-1,k-j-1}\cup\mathcal{Y}_{k-j-1}\left(n-j-1,\frac{j+1}{n}\right). 
\end{align*}
\end{proof}


\subsection{Applications to random topology}

Theorem \ref{Link} has several applications to random topology. Indeed, Garland's method \cite{Garland} and its refinements (see \cite{Zuk1}, \cite{Zuk2}, \cite{Opp1}, \cite{Opp2}), \.Zuk's criterion among them, have proven to be very effective tools for extracting global information about a pure $k$--dimensional simplicial complex using only information found in the $k-2$--dimensional links of the complex. 
\begin{theorem}[Garland's method]
Let $X$ be a pure $k$--dimensional simplicial complex. Suppose that, for all $(k-2)$--faces $\tau\in X$, we have  $\lambda^{(0)}(\operatorname{Link}(\tau,X))\geq1-\varepsilon>0$. Then $\lambda^{(k-1)}(X)\geq1-k\varepsilon.$
\end{theorem}
\begin{theorem}[\.Zuk's criterion]
Let $X$ be a pure $2$--dimensional simplicial complex. Suppose that, for all $0$--faces $\tau\in X$, we have that $\lambda^{(0)}(\operatorname{Link}(\tau,X))>1/2$ and $\operatorname{Link}(\tau,X)$ is connected. Then the fundamental group $\pi_1(X)$ has Kazhdan's property $(T)$.
\end{theorem}
The $\lambda^{(k-1)}(X)$ mentioned in the above statement of Garland's method refers to the smallest nonzero eigenvalue of the top--dimensional up--down Laplacian of $X$ under a particular weighted inner product (see \cite{Lub} or \cite{Eig} for greater detail). This eigenvalue will be referred to as the \emph{spectral gap} of $X$. In the special case where $k=1$ and $X$ is a connected graph, $\lambda^{(0)}(X)$ corresponds to the second smallest eigenvalue of the reduced Laplacian of $X$. Fortunately, the spectral gap of a random graph, in one form or another, is already quite well studied \cite{FO}, \cite{coja}, \cite{oliveira}, \cite{dense}, \cite{HKP}. Combining our characterization of the $k-2$--dimensional links in $\mathcal{T}_{n,k}$ with the best known  techniques for the kind of spectral gap estimation we would like to do, we find the following: 
\begin{proposition}\label{othermainprop}
Let $\delta>0$ be an arbitrary fixed constant, and let $\mathcal{X}$ be the union of $\lceil\delta\log n\rceil$ jointly independent copies of $\mathcal{T}_{n,k}$ with $k$ fixed. Then, for any fixed $s>0$, we have with probability $1-o(n^{-s})$ that $\lambda^{(k-1)}(\mathcal{X})=1-O\left(1/\sqrt{\log n}\right)$.    
\end{proposition}

\begin{proof} 
This follows immediately from Garland's method in combination with Lemma \ref{forapp} and a union bound on the probability that any $k-2$--dimensional link of $\mathcal{X}$ fails to meet the criteria of Garland's method. 
\end{proof} 

Applying \.Zuk's criterion in the case $k=2$ gives the following corollary: 
\begin{corollary}\label{mainprop}
Let $\delta>0$ be an arbitrary fixed constant, and let $\mathcal{X}$ be the union of $\lceil\delta\log n\rceil$ jointly independent copies of $\mathcal{T}_{n,2}$. Then, for any fixed $s>0$, we have with probability $1-o(n^{-s})$ that $\pi_1(\mathcal{X})$ has Kazhdan's property $(T)$.     
\end{corollary}

Particularly in computer science, there is a growing interest in generating families of graphs with large spectral gap, which are usually called expander graphs, and probabilistic constructions have been offered as a way to do this quickly and successfully with exceedingly high probability. For example, the authors of \cite{HKP} prove that the Erd\"os--R\'enyi random graph $\mathcal{G}(n,p)$ (in particular the random infinite family $\{\mathcal{G}(n,p(n))\}_{n\geq1}$) achieves the same spectral gap as found in Lemma \ref{forapp} and with equally high probability when $np\geq(1/2+\delta)\log n$ for any fixed $\delta>0$. Moreover, they show that the assumption $\delta>0$ is necessary for this to hold. The best known expander graph families have been constructed explicitly and have asymptotically (with respect to vertex count) constant average degree. So, while the Erd\"os--R\'enyi random graph can succeed at being a reliable expander, it is only able to do so if it is allowed an expected average degree far exceeding $\frac{1}{2}\log n$. 

Our characterization of the $k-2$--dimensional link of a determinantal $(n,k)$--tree as the union of a determinantal $(n-1,1)$--tree with an independent $\mathcal{G}(n-1,(k-1)/n)$ makes it, or perhaps the union of a small number of independent copies of it, a potential naturally occurring candidate for a random expander graph with constant expected average degree. Lemma \ref{forapp} shows that the number of superimposed independent copies of this graph required to match the result for $\G(n,p)$ is no more than $\delta\log n$ for every $\delta>0$. In particular, the resulting graph has an expected average degree of around $(k+1)\delta\log n$, improving upon $\mathcal{G}(n,p)$'s necessary expected average degree by a factor of arbitrary finite size. 
\subsection{Outline}
Section 2 is primarily devoted to establishing notation and basic homological definitions as well as providing background for the deterministic study of simplicial spanning trees. The only truly novel result of this section is Theorem \ref{submain}. 
In section 3 the results of section 2 are applied to the spanning trees of the $n-1$--simplex to give Theorem \ref{main} as well as several other new results about $\nu_{n,k}$. Section 4 establishes the tools required to prove Proposition \ref{othermainprop} and Corollary \ref{mainprop}. 

\section{Homological trees: simplicial and relative} 
A \emph{chain complex} is a sequence of abelian groups, $\{\mathcal{C}_j\}_{j\in\Z}$, called \emph{chain groups}, linked by group homomorphisms $\partial_j:\mathcal{C}_j\to\mathcal{C}_{j-1}$ called \emph{boundary maps} which satisfy $\partial_j\partial_{j+1}=0$ for all $j\in\Z$, or equivalently $\ker\partial_j\supseteq\Image\partial_{j+1}$. The \emph{$j$th homology group} of a chain complex such as this is defined to be the quotient group $\ker\partial_j/\Image\partial_{j+1}$. Since we will only be considering finitely--generated free $\Z$--modules for our chain groups, we can represent these boundary maps by integer matrices. By the structure theorem for finitely--generated abelian groups, the $j$th homology group can be expressed as the direct sum of a free abelian group, which is isomorphic to $\Z^{\beta_j}$ and called its \emph{free part}, and a finite abelian group called its \emph{torsion subgroup} which is a direct sum of finite cyclic groups. The rank $\beta_j$ of the free part is called the \emph{$j$th Betti number}. We will require the following standard fact from homological algebra which gives two equivalent formulas for what is commonly called the \emph{Euler characteristic} of a chain complex.   
\begin{lemma}\label{homalg}
Suppose $(\mathcal{C}_{\#},\partial_{\#})$ is a chain complex such that each chain group is freely and finitely generated and only finitely many of the chain groups are nontrivial. Let $f_j$ denote the rank of $\mathcal{C}_j$ for each $j\in\Z$. Then
$$\sum_{j\in\Z}(-1)^jf_j=\sum_{j\in\Z}(-1)^j\beta_j.$$
\end{lemma}

\subsection{Simplicial}

Fix an integer $n\geq1$. The set $[n]:=\{1,2,\dots,n\}$ will be our vertex set. 
For $-1\leq j\leq n-1$, a \emph{$j$--dimensional abstract simplex}, or \emph{$j$--face}, is a subset of $[n]$ with cardinality $j+1$. We denote the set of all $j$--faces on $[n]$ by ${[n]\choose{j+1}}$. All faces will be oriented according to the usual ordering on $[n]$. As such, we will be denoting elements of ${[n]\choose{j+1}}$ by $\{\tau_0,\tau_1,\dots,\tau_j\}$, where it is to be implicitly understood that $1\leq\tau_0<\tau_1<\dots<\tau_j\leq n$. Let $\partial=\partial_k^{[n]}$ be the matrix that's rows and columns are indexed respectively by ${[n]\choose k}$ and ${[n]\choose{k+1}}$, and for which, given $\sigma\in{[n]\choose k}$ and $\tau\in{[n]\choose{k+1}}$, we set  \begin{equation}\label{entries}
\partial(\sigma,\tau):=
\begin{cases}
(-1)^j, &\sigma=\tau\setminus\{\tau_j\} \\
0,&\text{otherwise}
\end{cases}
\end{equation} 
to account for our choice of orientation for each face. In particular, $\partial_0^{[n]}$ is the ${[n]\choose0}\times{[n]\choose1}$ all--ones matrix. 

An \emph{abstract simplicial complex with vertices in $[n]$} is a nonempty subset $X\subseteq\bigcup_{j\geq-1}{[n]\choose{j+1}}$ which, for every pair of subsets $\sigma\subseteq\tau$, satisfies $\tau\in X\implies\sigma\in X$. Let $\mathscr{A}_n$ denote the set of all abstract simplicial complexes on $[n]$. It is easily verified with this definition that $\mathscr{A}_n$ is closed under intersection and union. We write $X_j:={[n]\choose{j+1}}\cap X$ and define the \emph{dimension} of $X$ to be $\dim X:=\sup\{k\geq-1:X_k\neq\emptyset\}$. If $X$ has dimension $k$, $X$ is said to be \emph{pure} if, for every $-1\leq j<k$ and every $j$--face $\sigma\in X$, there exists a $\tau\in X_k$ such that $\sigma\subset\tau$. 
An important example is $K^k_n:=\bigcup_{j=0}^{k+1}{[n]\choose j}$, the complete $k$--dimensional complex on $[n]$. 

Given a matrix $M$ with entries indexed over the set $S\times T$, for $A\subseteq S$ and $B\subseteq T$, we write $M_{A,B}$ to denote the submatrix of $M$ with rows indexed by $A$ and columns indexed by $B$. We will also occasionally write $M_{\bullet,B}$ for the case $A=S$. As a point of clarification for this notation, transposes are handled by the convention of writing $M_{A,B}^t$ to mean $(M_{A,B})^t=(M^t)_{B,A}$. 

Given $X\in\mathscr{A}_n$, we define its chain complex $(\mathcal{C}_{\#},\partial_{\#})$ by $\mathcal{C}_j(X):=\Z^{X_j}$, and $\partial_j^X:=\partial_{X_{j-1},X_j}$. 
Note then that we have $\mathcal{C}_j(X)=0$
for all $j\geq n$ and all $j<-1$---because $|X_{-1}|=|\{\emptyset\}|=1$. We denote the $j$th homology group of this chain complex by $\tilde H_j(X)$. It is equivalent to the $j$th \emph{reduced} simplicial homology group, hence the notation. Note that, with this chain sequence, it makes sense to consider $\tilde{H}_{-1}=\ker\partial^{[n]}_{-1}/\Image\partial^{[n]}_0$ as well, although we can see that this group is always trivial since $\partial^{[n]}_0$ is surjective. 
\begin{lemma}\label{Hodgey}
$\partial^{[n]}_k\partial^{[n]t}_k+\partial^{[n]t}_{k-1}\partial^{[n]}_{k-1}=n\Id$. 
\end{lemma}
\begin{proof}
This follows by explicit computation of matrix entries via (\ref{entries}) combined with the chain sequence identity $\partial^{[n]}_{k-1}\partial^{[n]}_k=0$ which can also be seen to hold by explicit computation of matrix entries via (\ref{entries}). A more detailed explanation along these lines for the cases $k\geq2$ can be found in the proof of Lemma 3 in \cite{Kalai}, and the case $k=0$ is straightforward. For the case $k=1$, $\partial^{[n]}_1\partial^{[n]t}_1$ is the combinatorial Laplacian of the complete graph on $[n]$. Hence $\partial^{[n]}_1\partial^{[n]t}_1=n\Id-\operatorname{J}$ where $\operatorname{J}$ is the $n\times n$ all--ones matrix. Noticing that $\partial^{[n]t}_0\partial^{[n]}_0=\operatorname{J}$ then completes the proof. 
\end{proof}

Fix a $\Delta\in\mathscr{A}_n$ and set 
$\mathscr{C}_{n,k}(\Delta):=\left\{X\in\mathscr{A}_n:K_n^{k-1}\cap\Delta\subseteq X\subseteq K_n^k\cap\Delta\right\}$. We will call elements of this set the $(n,k)$--complexes of $\Delta$, or just $k$--complexes when the context is clear. Note this definition makes sense even if $\dim\Delta>k$, but in this case $\mathscr{C}_{n,k}(\Delta)=\mathscr{C}_{n,k}(K_n^k\cap\Delta)$. 
Building upon work done by Duval, Klivans, and Martin \cite{STT}, \cite{DKM1}, \cite{DKM2}, Bernardi and Klivans \cite{BK} define a higher--dimensional forest of $\Delta$ to be a subset $F\subseteq\Delta_k$ such that $\partial_{\Delta_{k-1},F}$ is injective. Our definition of a forest will be equivalent to this one except that we will consider $F$ to be an entire element of $\mathscr{C}_{n,k}(\Delta)$ by adding to it the $k-1$--skeleton of $\Delta$, more similar to the situation in \cite{STT}. That is to say, an $F\in\mathscr{C}_{n,k}(\Delta)$ is called a $k$--\emph{forest} of $\Delta$ if $\beta_k(F)=0$, in which case we write $F\in\mathscr{F}_{n,k}(\Delta)$. It is easily verified that this is equivalent to Definition 3 in \cite{BK} by noting that $\tilde{H}_k(F)=\ker\partial_k^F$ is a free group and thus is 0 if and only if its rank is 0. A $k$--forest $T\in\mathscr{F}_{n,k}(\Delta)$ with $|T_k|=\rank\partial_k^{\Delta}$ (the maximal possible value for a forest) is said to be a \emph{$k$--tree of $\Delta$}. Let $\mathscr{T}_{n,k}(\Delta)$ denote the set of $k$--trees of $\Delta$---\cite{BK} calls these the \emph{spanning forests} of $\Delta$. In the case $\Delta=K_n^{n-1}$, or just $\Delta\supseteq K_n^k$, these are Kalai's $k$--dimensional $\Q$--acyclic simplicial complexes mentioned in the introduction, and in this case we will suppress the $(\Delta)$ in all the above notation. The following result from \cite{DKM2} generalizes Proposition 2 from \cite{Kalai} to general $k$--trees of $\Delta\in\mathscr{A}_n$: \begin{lemma}\label{general2/3}
For $X\in\mathscr{C}_{n,k}(\Delta)$, if any two of the following conditions hold, then so does the other condition and moreover $X\in\mathscr{T}_{n,k}(\Delta)$:
\begin{itemize}
    \item $|X_k|=\rank\partial_k^{\Delta}$,
    \item $\beta_{k-1}(X)=\beta_{k-1}(\Delta)$, 
    \item $\beta_k(X)=0$. 
\end{itemize}
\end{lemma}

\subsection{Relative}

Given a pair $(X,Y)\in\mathscr{A}_n^2$ with $X\supseteq Y$, we can define another chain complex by  $\mathcal{C}_j(X,Y):=\Z^{X_j\setminus Y_j}$ with boundary maps $\partial_j^{X/Y}:=\partial_{X_{j-1}\setminus Y_{j-1},X_j\setminus Y_j}$. The homology groups of this chain complex are called \emph{relative homology groups} and denoted $H_j(X,Y)$. We can recover ordinary reduced homology from this by taking $Y$ to be $\{\emptyset\}$ or any complex with a single 0--face and no larger faces. Moreover, since homology is homotopy invariant, we also have that $H_j(X,Y)\cong\tilde H_j(X)$ as long as $Y$ is contractible. 

An example of a contractible  complex, the simplicial cone of a complex $\Delta\in\mathscr{A}_{n-1}$ is defined by $\operatorname{Cone}(n,\Delta):=\Delta\cup\{\sigma\cup\{n\}:\sigma\in\Delta\}\in\mathscr{A}_n$. We also define for any $X\in\mathscr{C}_{n,k}(\operatorname{Cone}(n,\Delta))$ the complexes $\operatorname{Proj}(n,X):=X\cap\Delta\in\mathscr{C}_{n-1,k}(\Delta)$ and
$\operatorname{Link}(n,X):=\{\sigma\in\Delta:\sigma\cup\{n\}\in X\}\in\mathscr{C}_{n-1,k-1}(\Delta)$. The next lemma defines what we will call the \emph{binomial correspondence}
$\mathscr{C}_{n,k}(\operatorname{Cone}(n,\Delta))\xrightarrow{\sim}\mathscr{C}_{n-1,k}(\Delta)\times\mathscr{C}_{n-1,k-1}(\Delta)$ for its relationship to the binomial recurrence formula ${n\choose{k+1}}={{n-1}\choose {k+1}}+{{n-1}\choose k}$ in the special case that $\operatorname{Cone}(n,\Delta)\supseteq K_n^k$. 

\begin{lemma}\label{binomial}
For $\Delta\in\mathscr{A}_{n-1}$, the map $\varphi(X):=(\operatorname{Proj}(n,X),\operatorname{Link}(n,X))$ is an injection from $\mathscr{C}_{n,k}(\operatorname{Cone}(n,\Delta))$ into $\mathscr{C}_{n-1,k}(\Delta)\times\mathscr{C}_{n-1,k-1}(\Delta)$, and its (left) inverse is given by $\varphi^{-1}(F,R)=F\cup\operatorname{Cone}(n,R)$ which extends to an injection of $\mathscr{C}_{n-1,k}(\Delta)\times\mathscr{C}_{n-1,k-1}(\Delta)$ into $\mathscr{C}_{n,k}(\operatorname{Cone}(n,\Delta))$. 
\end{lemma}

\begin{proof}
For brevity, set $\Delta_{(j)}:=K_n^j\cap\Delta$. Since $\varphi^{-1}$ is clearly also a right inverse of $\varphi$, the only thing to show is that the images of these restrictions lie where we claim they do. 
Starting with $\varphi$, we first recall that $$\mathscr{C}_{n,k}(\operatorname{Cone}(n,\Delta))=\mathscr{C}_{n,k}(\operatorname{Cone}(n,\Delta)_{(k)})=\mathscr{C}_{n,k}(\Delta_{(k)}\cup\{\sigma\cup\{n\}:\sigma\in\Delta_{(k-1)}\}).$$
So for $X\in\mathscr{C}_{n,k}(\operatorname{Cone}(n,\Delta))$ we have $$\Delta_{(k-1)}\cup\{\sigma\cup\{n\}:\sigma\in\Delta_{(k-2)}\}\subseteq X\subseteq\Delta_{(k)}\cup\{\sigma\cup\{n\}:\sigma\in\Delta_{(k-1)}\}.$$
Thus $\Delta_{(k-1)}\subseteq\operatorname{Proj}(n,X)\subseteq\Delta_{(k)}$ and $\Delta_{(k-2)}\subseteq\operatorname{Link}(n,X)\subseteq\Delta_{(k-1)}$ as desired. 
As for $\varphi^{-1}$, we have $$\operatorname{Cone}(n,\Delta)_{(k-1)}=\Delta_{(k-1)}\cup\operatorname{Cone}(n,\Delta_{(k-2)})\subseteq F\cup\operatorname{Cone}(n,R)\subseteq\operatorname{Cone}(n,\Delta)_{(k)}$$
due to the assumption that $(F,R)\in\mathscr{C}_{n-1,k}(\Delta)\times\mathscr{C}_{n-1,k-1}(\Delta)$. 
\end{proof}

\begin{theorem}[Excision Theorem]
For any $A,B\in\mathscr{A}_n$, we have $H_*(A,A\cap B)\cong H_*(A\cup B,B)$. 
\end{theorem}

\begin{theorem}[Long exact sequence of a triple]
For $A\subseteq B\subseteq C\in\mathscr{A}_n$, arrows exist for which the following sequence is exact:
$$\begin{tikzcd}
\cdots\arrow{r}&H_k(B,A)\arrow{r}&H_k(C,A)\arrow{r}&H_k(C,B)\arrow{dll} \\ &H_{k-1}(B,A)\arrow{r}&\cdots\arrow{r}&H_0(C,B)\arrow{r}&0
\end{tikzcd}.$$
That is, the kernel of each arrow in this diagram is equal to the image of the preceding arrow. 
\end{theorem}


\begin{lemma}\label{evenmoregeneral2/3}
For any $\Delta\in\mathscr{A}_n$, $X\in\mathscr{C}_{n,k}(\Delta)$, and $Y\in\mathscr{C}_{n,j}(X)$, any two of the following conditions imply the other: \begin{itemize}
    \item $|X_k|-|Y_k|=\rank\partial_k^{\Delta/Y}$,
    \item $\beta_{k-1}(X,Y)=\beta_{k-1}(\Delta,Y)$, 
    \item $\beta_k(X,Y)=0$.
\end{itemize}
\end{lemma}
\begin{proof}
The long exact sequence of the triple $(\Delta,X,Y)$ contains the sequence
$$\begin{tikzcd}
&0\arrow{r}&H_{k+1}(\Delta,Y)\arrow{r}&H_{k+1}(\Delta,X)\arrow{dll} \\ 
&H_k(X,Y)\arrow{r}&H_k(\Delta,Y)\arrow{r}&H_k(\Delta,X)\arrow{dll} \\ 
&H_{k-1}(X,Y)\arrow{r}&H_{k-1}(\Delta,Y)\arrow{r}&0. 
\end{tikzcd}$$


The exactness of this sequence implies that 
\begin{equation}\label{exact}
\beta_k(X,Y)+\big(\beta_{k+1}(\Delta,Y)-\beta_{k+1}(\Delta,X)-\beta_k(\Delta,Y)+\beta_k(\Delta,X)\big)-\big(\beta_{k-1}(X,Y)-\beta_{k-1}(\Delta,Y)\big)=0.
\end{equation} 

It suffices to show that the  first bulleted condition is equivalent to the vanishing of the middle term in (\ref{exact}). Indeed, by the definition of relative homology, the rank--nullity theorem, dimensional considerations, and the fact that $\Delta,X$ have the same $k-1$--skeleton, we have 
\begin{align*}
\beta_k(\Delta,X)-\beta_{k+1}(\Delta,X)&=\dim\ker\partial_k^{\Delta/X}-(\rank\partial_{k+1}^{\Delta/X}+\dim\ker\partial_{k+1}^{\Delta/X})+\rank\partial_{k+2}^{\Delta/X} \\ 
&=\dim\ker\partial_k^{\Delta/X}-|(\Delta/X)_{k+1}|+\rank\partial_{k+2}^{\Delta} \\ 
&=|(\Delta/X)_k|-|\Delta_{k+1}|+\rank\partial_{k+2}^{\Delta}. 
\end{align*}
By most of the same considerations, $\beta_k(\Delta,Y)-\beta_{k+1}(\Delta,Y)=\dim\ker\partial_k^{\Delta/Y}-|\Delta_{k+1}|+\rank\partial_{k+2}^{\Delta}$. So 
\begin{align*}
&\beta_{k+1}(\Delta,Y)-\beta_{k+1}(\Delta,X)-\beta_k(\Delta,Y)+\beta_k(\Delta,X) \\ 
=&|\Delta_k|-|X_k|-\dim\ker\partial_k^{\Delta/Y} \\ 
=&|Y_k|+(|(\Delta/Y)_k|-\dim\ker\partial_k^{\Delta/Y})-|X_k| \\ 
=&|Y_k|+\rank\partial^{\Delta/Y}_k-|X_k|
\end{align*}
is equal to 0 if and only if the first bulleted condition holds. 
The result therefore follows from (\ref{exact}). 
\end{proof}

Whenever two, and therefore all three, of the conditions of this lemma hold, we will call $(X,Y)$ a \emph{relative $(n,k)$--forest of $\Delta$}. We will denote the set of such pairs by  $\F_{n,k}^{\operatorname{rel}}(\Delta)$. 
In the special case that $Y$ is contractible (or just that $\dim Y<k-2$) we get that $X\in\mathscr{T}_{n,k}(\Delta)$ if and only if $(X,Y)\in\mathscr{F}_{n,k}^{\operatorname{rel}}(\Delta)$. We therefore have the following corollary which gives the original \cite{Kalai} three equivalent necessary and sufficient pairs of conditions for a $k$--complex of a homologically connected $\Delta$ (i.e. $\beta_{k-1}(\Delta)=0$) to be a $k$--tree of $\Delta$: 
\begin{corollary}\label{2/3} 
For any $\Delta\in\mathscr{A}_n$ with $\beta_{k-1}(\Delta)=0$ and $X\in\mathscr{C}_{n,k}(\Delta)$, any two of the following conditions imply the other, and in particular are equivalent to the statement that $X\in\mathscr{T}_{n,k}(\Delta)$:
\begin{itemize}
    \item $|X_k|={{n-1}\choose k}$, 
    \item $\beta_k(X)=0$,
    \item $\beta_{k-1}(X)=0$.
\end{itemize}
\end{corollary}

An $R\in\mathscr{C}_{n,k-1}(\Delta)$ is called a \emph{$k-1$--root of $\Delta$} if $\partial^{\Delta/R}_k$ is surjective and has full rank (\cite{BK}, Definition 9). 
Whenever any two of the conditions in the following lemma hold for a pair $(X,Y)\in\mathscr{C}_{n,k}(\Delta)\times\mathscr{C}_{n,k-1}(\Delta)$, we will call $(X,Y)$ a \emph{rooted $(n,k)$--forest of $\Delta$} and write $(X,Y)\in\mathscr{F}_{n,k}^{\operatorname{root}}(\Delta)$. Equivalently, a rooted $(n,k)$--forest of $\Delta$ is a pair $(X,Y)\in\mathscr{F}_{n,k}(\Delta)\times\mathscr{C}_{n,k-1}(\Delta)$ such that $Y$ is a $k-1$--root of $X$ (\cite{BK}, Definition 12). 
\begin{lemma}\label{moregeneral2/3}
For $X\in\mathscr{C}_{n,k}(\Delta)$ and $Y\in\mathscr{C}_{n,k-1}(\Delta)$, any two of the following conditions imply the other:
\begin{itemize}
    \item $|X_k|=|X_{k-1}\setminus Y_{k-1}|$, 
    \item $\beta_{k-1}(X,Y)=0$, 
    \item $\beta_k(X,Y)=0$.
\end{itemize}
\end{lemma}

\begin{proof}
By Lemma \ref{homalg}, we have
$\sum_{j=-1}^k(-1)^j(|X_j\setminus Y_j|)=\sum_{j=-1}^k(-1)^j\beta_j(X,Y)$, which simplifies to 
$$|X_k|-|X_{k-1}\setminus Y_{k-1}|=\beta_k(X,Y)-\beta_{k-1}(X,Y).$$
Indeed, since $Y_j=X_j=\Delta_j$ for all $j<k-1$, we have $\mathcal{C}_j(X,Y)=0$ for all $j<k-1$. As $H_j(X,Y)$ is a quotient of a subgroup of $\mathcal{C}_j(X,Y)$, the former vanishes with the latter. 
\end{proof} 

\begin{corollary}
For any $\Delta\in\mathscr{A}_n$, we have $$\mathscr{F}_{n,k}^{\operatorname{root}}(\Delta)=\left\{(X,Y)\in\mathscr{F}_{n,k}^{\operatorname{rel}}(\Delta)\cap\big(\mathscr{C}_{n,k}(\Delta)\times\mathscr{C}_{n,k-1}(\Delta)\big):\beta_{k-1}(\Delta,Y)=0\right\}.$$ 
\end{corollary}

\begin{proof}
Clearly $\mathscr{F}_{n,k}^{\operatorname{root}}(\Delta)\supseteq\{(X,Y)\in\mathscr{F}_{n,k}^{\operatorname{rel}}(\Delta)\cap\big(\mathscr{C}_{n,k}(\Delta)\times\mathscr{C}_{n,k-1}(\Delta)\big):\beta_{k-1}(\Delta,Y)=0\}$. For the other inclusion, suppose that $(X,Y)\in\mathscr{F}_{n,k}^{\operatorname{root}}(\Delta)$. Then the long exact sequence in the proof of Lemma \ref{evenmoregeneral2/3} implies that $0=\beta_{k-1}(X,Y)\geq\beta_{k-1}(\Delta,Y)\geq0$. Thus $(X,Y)\in\mathscr{F}_{n,k}^{\operatorname{rel}}(\Delta)$ since $\beta_{k-1}(X,Y)=0=\beta_{k-1}(\Delta,Y)$ and $\beta_k(X,Y)=0$. 
\end{proof}




\begin{lemma}\label{ABD}
Suppose $\Delta\in\mathscr{A}_n$ and that $F\in\mathscr{C}_{n,k}(\Delta)$, $R\in\mathscr{C}_{n,k-1}(\Delta)$ satisfy $|F_k|=|\Delta_{k-1}\setminus R_{k-1}|$, and write $\overline{R}:=\Delta_{k-1}\setminus R_{k-1}$. Then 
$\det\partial_{\overline{R},F_k}\neq0\iff|\det\partial_{\overline{R},F_k}|=|H_{k-1}(F,R)|\iff(F,R)\in\mathscr{F}_{n,k}^{\operatorname{root}}(\Delta)$. 
\end{lemma}
\begin{proof}
The relative chain complex for the pair $(F,R)$ is 
$\begin{tikzcd}
0\arrow{r}&\mathcal{C}_k(F)\arrow{r}{\partial_{\overline{R},F_k}}&\mathcal{C}_{k-1}(\overline{R})\arrow{r}&0\arrow{r}&\cdots
\end{tikzcd}$ since $R_k=\emptyset$, $F_{k-1}\setminus R_{k-1}=\Delta_{k-1}\setminus R_{k-1}=\overline{R}$, and $F_{k-2}=R_{k-2}$. We therefore have $H_{k-1}(F,R)=\Z^{\overline{R}}/\partial_{\overline{R},F_k}\Z^{F_k}$ and so, provided $\det\partial_{\overline{R},F_k}\neq0$, we have $|H_{k-1}(F,R)|=|\det\partial_{\overline{R},F_k}|$---this is seen most easily by putting  $\partial_{\overline{R},F_k}$ in Smith normal form. Having $|H_{k-1}(F,R)|=|\det\partial_{\overline{R},F_k}|$ clearly implies $|H_{k-1}(F,R)|<\infty$ which, by the previous lemma, implies that $(F,R)\in\mathscr{F}_{n,k}^{\operatorname{root}}(\Delta)$. Finally, $(F,R)\in\mathscr{F}_{n,k}^{\operatorname{root}}(\Delta)\implies\beta_k(F,R)=0\implies \ker\partial_{\overline{R},F_k}=H_k(F,R)=0\implies\det\partial_{\overline{R},F_k}\neq0$. 
\end{proof}
Essentially all of the content of this last lemma is proven in greater detail in Lemmas 14 and 17 of \cite{BK}. 

\begin{theorem}\label{submain}
Set $T=F\cup\operatorname{Cone}(n,R)$ where $F\in\mathscr{C}_{n-1,k}(\Delta)$ and $R\in\mathscr{C}_{n-1,k-1}(\Delta)$ for any $\Delta\in\mathscr{A}_{n-1}$. Then $\tilde H_*(T)\cong H_*(F,R)$, and $T\in\mathscr{T}_{n,k}(\operatorname{Cone}(n,\Delta))$ if and only if $(F,R)\in\mathscr{F}_{n-1,k}^{\operatorname{root}}(\Delta)$. Moreover the binomial correspondence restricts to a bijection $\varphi:\mathscr{T}_{n,k}(\operatorname{Cone}(n,\Delta))\xrightarrow{\sim}\mathscr{F}_{n-1,k}^{\operatorname{root}}(\Delta)$. \end{theorem}

\begin{proof}
Since $F$ has no $k$--faces containing $n$ and $R\subseteq F$, we have $F\cap\operatorname{Cone}(n,R)=R$. Thus, by excision, the pairs $(T,\operatorname{Cone}(n,R))$ and $(F,R)$ have the same relative homology in every dimension. But, since $\operatorname{Cone}(n,R)$ is contractible, $\tilde H_*(T)\cong H_*(F,R)$. Therefore $T\in\mathscr{T}_{n,k}(\operatorname{Cone}(n,\Delta))$ if and only if $(F,R)\in\mathscr{F}_{n-1,k}^{\operatorname{root}}(\Delta)$ by Lemmas \ref{general2/3} and \ref{moregeneral2/3}.
The bijection then follows from Lemma \ref{binomial}.  
\end{proof}

\begin{corollary}\label{maincor}
Set $T=F\cup\operatorname{Cone}(n,R)$, where $F\in\mathscr{C}_{n-1,k}$ and $R\in\mathscr{C}_{n-1,k-1}$. Then $\tilde H_*(T)\cong H_*(F,R)$, and $T\in\mathscr{T}_{n,k}$ if and only if $(F,R)\in\mathscr{F}_{n-1,k}^{\operatorname{root}}$. Moreover the binomial correspondence restricts to a bijection $\varphi:\mathscr{T}_{n,k}\xrightarrow{\sim}\mathscr{F}_{n-1,k}^{\operatorname{root}}$. 
\end{corollary}

\begin{proof}
This follows by recalling that $\mathscr{T}_{n,k}(\operatorname{Cone}(n,\Delta))=\mathscr{T}_{n,k}((\Delta\cap K_{n-1}^k)\cup\{\sigma\cup\{n\}:\sigma\in\Delta\cap K_{n-1}^{k-1}\})$. Thus if $\Delta\supseteq K_{n-1}^k$, then  $(\Delta\cap K_{n-1}^k)\cup\{\sigma\cup\{n\}:\sigma\in\Delta\cap K_{n-1}^{k-1}\}=K_n^k$. 
\end{proof}


\section{Determinantal measure}


For ease of reading, we will for this section abuse notation by identifying each $j$--complex of $K_n^{n-1}$ with it's set of $j$--dimensional faces. For any $X\in\mathscr{C}_{n,j}$, we also will let $\overline{X}$ denote the $j$--complex that's set of $j$--faces is the complement (with respect to ${[n]\choose{j+1}}$) of $X_j$. Define the submatrix $\widehat\partial$ of $\partial$ by deleting all rows of $\partial$ that correspond to elements of ${[n]\choose k}$ which contain the vertex $n$ (thus $\widehat\partial_0^{[n]}=\partial_0^{[n]}$). We have the following corollary of Lemma \ref{ABD}:

\begin{corollary}\label{ABDlight}
Suppose $T\in\mathscr{C}_{n,k}$ satisfies $|T_k|={{n-1}\choose k}$. Then
$\det\widehat\partial_{\bullet,T}\neq0\iff|\det\widehat\partial_{\bullet,T}|=|\tilde{H}_{k-1}(T)|\iff T\in\mathscr{T}_{n,k}$. \end{corollary}

\begin{proof}
Let $R=\operatorname{Cone}\left(n,K_{n-1}^{k-2}\right)$ so that $|T_k|+|R_{k-1}|={n\choose k}$, and  $\partial_{\overline{R},T}=\widehat\partial_{\bullet,T}$. This now follows from Corollary \ref{2/3} and Lemma \ref{ABD}. 
\end{proof}

This last corollary was originally proven for the cases $k\geq1$ (\cite{Kalai}, Lemma 2) by Kalai, who combined this with the Cauchy--Binet formula and some deft linear algebra to show (\cite{Kalai}, Theorem 1) that
$$n^{{n-2}\choose k}=\det\widehat\partial {\widehat\partial}^t=\sum_{T\in\mathscr{T}_{n,k}}\det\widehat\partial_{\bullet,T}^2=\sum_{T\in\mathscr{T}_{n,k}}|\tilde{H}_{k-1}(T)|^2.$$
Understanding that $\widehat\partial_0^{[n]}=\partial_0^{[n]}$ and $\tilde H_{-1}(T)=0$, the case $k=0$ can also be seen to hold. This gives us a natural probability measure $\nu=\nu_{n,k}$ on $\mathscr{T}_{n,k}$. Namely, 
$$\nu_{n,k}(T):=\frac{\det\widehat\partial_{\bullet,T}^2}{\det\widehat\partial {\widehat\partial}^t}=\frac{|\tilde H_{k-1}(T)|^2}{n^{{n-2}\choose k}}.$$
This measure was originally formulated in greater generality by Lyons (\cite{Lyons}, \S12) and was further expanded upon in \cite{Lyons2}. This version of the measure has also been considered again recently by authors such as Kahle and Newman \cite{KN} and  M\'esz\'aros \cite{Meszaros}. 

Measures defined in the above manner are said to be \emph{determinantal} \cite{Lyons}, as are the random variables associated to them. The following lemma is a special case of Theorem 5.1 from \cite{Lyons}: 
\begin{lemma}\label{incl}
Let $R,S$ be finite sets and let $M$ be an $R\times S$ matrix of rank $|R|$. Let $\mu$ be the determinantal measure on $S$ which is defined by $\mu(T)=\frac{\det M_{\bullet,T}^2}{\det(MM^t)}$ for all $T\subseteq S$ of size $|R|$. Let $P:=M^t(MM^t)^{-1}M$ (this is the matrix of the projection onto the rowspace of $M$). Then, for any $B\subseteq S$, $$\mu(T:T\supseteq B)=\det P_{B,B}\quad\text{ and }\quad\mu(T:T\subseteq S\setminus B)=\det(\Id-P)_{B,B}.$$
\end{lemma}
Determinantal random variables enjoy the negative associations property (\cite{Lyons}, Theorem 6.5), which can be stated for random variables on $\mathscr{C}_{n,k}$ as follows: A function $f:\mathscr{C}_{n,k}\to\R$ is called \emph{increasing} if $f(X)\leq f(X\cup Y)$ for every $X,Y\in\mathscr{C}_{n,k}$. A random variable $\mathcal{X}\in\mathscr{C}_{n,k}$ is said to have \emph{negative associations} if, for every pair of increasing functions $f_1,f_2$ and every $Y\in\mathscr{C}_{n,k}$, we have $$\E\big[f_1(\mathcal{X}\cap Y)f_2(\mathcal{X}\cap \overline{Y})\big]\leq\E\big[f_1(\mathcal{X}\cap Y)\big]\E\big[f_2(\mathcal{X}\cap \overline{Y})\big].$$ 
We will make use of negative associations when we go to prove our applications to random topology. 

As we see from Lemma \ref{incl}, for any $B\in\mathscr{C}_{n,k}$ and $A\in\mathscr{C}_{n,k-1}$, we have 
\begin{equation}\label{clusions}
\nu_{n,k}\left(T:T\supseteq B\right)=\det (P_{n,k})_{B,B}\quad\text{ and }\quad\nu_{n,k-1}\left(T:T\subseteq \overline{A}\right)=\det(\Id- P_{n,k-1})_{A,A}
\end{equation}
where $P_{n,k}:=\widehat\partial^t(\widehat\partial\widehat\partial^{t})^{-1}\widehat\partial$. Our use of $\widehat\partial$ is actually a choice of basis for the rowspace of $\partial$, and an arbitrary one at that. Let $A$ be any $k-1$--root of $K_n^{n-1}$---this implies that $|A_{k-1}|={{n-1}\choose k}$ (\cite{Kalai}, Lemma 1). Then, by applying a change of basis, we also have the equivalent definition(s) \begin{equation}\label{alt}
\nu_{n,k}(T):=\frac{\det\partial_{A,T}^2}{\det(\partial \partial^t)_{A,A}}, \end{equation}
all of which correspond to the same $P_{n,k}$.
M{\'e}sz{\'a}ros (\cite{Meszaros}, Lemma 14) determined that  $$P_{n,k}:=\frac{1}{n}\partial^{[n]t}_k\partial_k^{[n]}.$$
By Lemma \ref{Hodgey}, we also have $\Id-P_{n,k}=\frac{1}{n}\partial^{[n]}_{k+1}\partial^{[n]t}_{k+1}=:Q_{n,k}$. By Lemma \ref{ABD} and Cauchy--Binet, we therefore have the following lemma:
\begin{lemma}\label{bridge}
Suppose $B\in\mathscr{C}_{n,k}$ and $A\in\mathscr{C}_{n,k-1}$. Then \begin{align*}
\nu_{n,k-1}\left(T:T\subseteq\overline{A}\right)=\det(Q_{n,k-1})_{A,A}=n^{-|\overline{A}|}&\sum_{B:(B,\overline{A})\in\mathscr{T}_{n,k}^{\operatorname{root}}}\det\partial^2_{A,B}, \\
\text{and}\quad\nu_{n,k}\left(T:T\supseteq B\right)=\det(P_{n,k})_{B,B}=n^{-|B|}&\sum_{A:(B,\overline{A})\in\mathscr{T}_{n,k}^{\operatorname{root}}}\det\partial_{A,B}^2.
\end{align*} 
\end{lemma} 

\begin{corollary}\label{split}
For any $T\in\mathscr{T}_{n,k}$, $T'\in\mathscr{T}_{n,k-1}$ we have 
$|H_{k-1}(T,T')|=|\tilde{H}_{k-1}(T)||\tilde{H}_{k-2}(T')|$. Moreover 
$$\left\{(F,E)\in\mathscr{F}_{n,k}^{\operatorname{root}}:|F|={{n-1}\choose k}\right\}=\mathscr{T}_{n,k}\times\mathscr{T}_{n,k-1}.$$
\end{corollary}
\begin{proof}
Let $A:=\overline{T'}$ so that, as in expression (\ref{alt}), $|A_{k-1}|={{n-1}\choose k}$. Then by Lemma \ref{incl}, we have 
$$\nu_{n,k}(T)=\frac{\det\partial_{A,T}^2}{\det(\partial\partial^t)_{A,A}}=\frac{\det\partial_{A,T}^2}{n^{{n-1}\choose k}\nu_{n,k-1}\left(\left\{S:S\subseteq \overline{A}\right\}\right)}=\frac{\det\partial_{A,T}^2}{n^{{n-1}\choose k}\nu_{n,k-1}(T')}.$$
Thus by Lemma \ref{ABD}, 
$$\frac{|H_{k-1}(T,T')|^2}{n^{{n-1}\choose k}}=\frac{\det\partial_{A,T}^2}{n^{{n-1}\choose k}}=\nu_{n,k}(T)\nu_{n,k-1}(T')=\frac{|\tilde{H}_{k-1}(T)|^2|\tilde{H}_{k-2}(T')|^2}{n^{{{n-2}\choose k}+{{n-2}\choose {k-1}}}}.$$
The set identity in the statement of the corollary now follows from Lemma \ref{ABD}, as the left hand side of this is finite if and only if the right hand side is finite, and $T,T'$ are also the correct sizes for the statement to hold---i.e., ${{n-1}\choose k}+{{n-1}\choose{k-1}}={n\choose k}$. 
\end{proof}

\begin{corollary}\label{newdef}
Suppose $T\in\mathscr{T}_{n,k}$, and let $F=\operatorname{Proj}(n,T)$ and $E=\operatorname{Link}(n,T)$. Then 
$$\nu_{n,k}(F\cup\operatorname{Cone}(n,E))=\frac{\det\partial_{\overline{E},F}^2}{n^{{{n-2}\choose k}}}=\frac{|H_{k-1}(F,E)|^2}{n^{{{n-2}\choose k}}}.$$ 
\end{corollary}
%
\begin{proof}
This follows from Corollary \ref{maincor}, the original definition $\nu_{n,k}(T)=\frac{|\tilde H_{k-1}(T)|^2}{n^{{n-2}\choose k}}$, and Lemma \ref{ABD}. 
\end{proof}

\begin{corollary}
Suppose $T\in\mathscr{T}_{n,k}$ is such that $T':=\operatorname{Proj}(n,T)\in\mathscr{T}_{n-1,k}$ and $T'':=\operatorname{Link}(n,T)\in\mathscr{T}_{n-1,k-1}$. Then 
$\nu_{n,k}(T)=\nu_{n-1,k}(T')\nu_{n-1,k-1}(T'')(1-1/n)^{{n-2}\choose k}$. 
\end{corollary}

\begin{proof}
By the previous corollary, Lemma \ref{ABD}, and Corollary \ref{split},  $$\nu_{n,k}(T)=\frac{|H_{k-1}(T',T'')|^2}{n^{{n-2}\choose k}}=\frac{|\tilde H_{k-1}(T')|^2}{(n-1)^{{n-3}\choose k}}\frac{|\tilde H(T'')|^2}{(n-1)^{{n-3}\choose{k-1}}}\frac{(n-1)^{{{n-3}\choose k}+{{n-3}\choose {k-1}}}}{n^{{n-2}\choose k}}.$$ 
\end{proof}

\begin{lemma}\label{incr} 
For all $F\in\mathscr{C}_{n-1,k}$ and $E\in\mathscr{C}_{n-1,k-1}$, we have 
\begin{align*} 
\nu_{n,k}(T:\operatorname{Proj}(n,T)=F)=\nu_{n-1,k}(T':T'\supseteq F)(1-1/n)^{|F|}n^{|F|-{{n-2}\choose k}}
\end{align*}
and 
\begin{align*}
\nu_{n,k}(T:\operatorname{Link}(n,T)=E)=\nu_{n-1,k-1}(T'':T''\subseteq E)(1-1/n)^{|\overline{E}|}n^{|\overline{E}|-{{n-2}\choose k}}.
\end{align*}
\end{lemma}
\begin{proof}
By Lemma \ref{bridge}, Corollary \ref{newdef}, and the Cauchy--Binet formula, we have  $$(n-1)^{|F|}\nu_{n-1,k}\left(T:T\supseteq F\right)=\det(\partial^t\partial)_{F,F}=n^{{n-2}\choose k}\nu_{n,k}(T:\operatorname{Proj}(n,T)=F)$$
and 
$$(n-1)^{|\overline{E}|}\nu_{n-1,k-1}\left(T:T\subseteq E\right)=\det(\partial\partial^t)_{\overline{E},\overline{E}}=n^{{n-2}\choose k}\nu_{n,k}(T:\operatorname{Link}(n,T)=E).$$ 
\end{proof}

\begin{corollary}
Using the notation from Theorem \ref{main}, there are couplings $\pi_{n,k}$ and $\lambda_{n,k}$ of $\mathcal{T}_{n,k}$, $\mathcal{T}_{n-1,k}$, $\mathcal{Y}_{n-1,k}$ and  $\mathcal{T}_{n,k}$, $\mathcal{T}_{n-1,k-1}$, $\mathcal{Y}_{n-1,k-1}$ respectively such that, marginally, $\mathcal{T}_{n-1,k},\mathcal{T}_{n-1,k-1}$ are independent of $\mathcal{Y}_{n-1,k},\mathcal{Y}_{n-1,k-1}$ respectively, and 
$$\pi_{n,k}\left((T,T',Y'):\operatorname{Proj}(n,T)=T'\setminus Y'\right)=1=\lambda_{n,k}\left((T,T'',Y''):\operatorname{Link}(n,T)=T''\cup Y''\right).$$
Namely, 
$$\pi_{n,k}(T,T',Y'):=\mu_{n-1,k}(Y')\nu_{n-1,k}(T')\frac{\nu_{n,k}(T)\one\{\operatorname{Proj}(n,T)=T'\setminus Y'\}}{\nu_{n,k}(S:\operatorname{Proj}(n,S)=T'\setminus Y')}$$ and 
$$\lambda_{n,k}(T,T'',Y''):=\mu_{n-1,k-1}(Y'')\nu_{n-1,k-1}(T'')\frac{\nu_{n,k}(T)\one\{\operatorname{Link}(n,T)=T''\cup Y''\}}{\nu_{n,k}(S:\operatorname{Link}(n,S)=T''\cup Y'')}.$$
\end{corollary}

\begin{proof}
It suffices to show that $\pi_{n,k}$ has the correct marginal densities, as the proof for $\lambda_{n,k}$ is basically identical. Summing over all $T$ clearly produces the independent coupling of $\mathcal{T}_{n-1,k}$, $\mathcal{Y}_{n-1,k}$. For the remaining marginal, Corollary \ref{incr} gives us that \begin{equation}\label{givenT}
\pi_{n,k}(T,T',Y')=\nu_{n,k}(T)\frac{\nu_{n-1,k}(T')\one\{T'\supseteq\operatorname{Proj}(n,T)\}}{\nu_{n-1,k}(S:S\supseteq\operatorname{Proj}(n,T))}\frac{\mu_{n-1,k}(Y')\one\{\operatorname{Proj}(n,T)=T'\setminus Y'\}}{\mu_{n-1,k}(S:T'\setminus S=\operatorname{Proj}(n,T))}
\end{equation}
which, summed over $Y'$ and then $T'$, gives the desired marginal. 
\end{proof}

\begin{proof}[Proof of Theorem \ref{main}]
It suffices to show that 
\begin{equation}\label{density}
(T,T',T'',Y',Y'')\mapsto\frac{\pi_{n,k}(T,T',Y')\lambda_{n,k}(T,T'',Y'')}{\nu_{n,k}(T)}
\end{equation}
is a probability density with the claimed marginal densities. Considering expression (\ref{givenT}), summing (\ref{density}) over $Y'$ and $T'$ gives $\lambda_{n,k}(T,T'',Y'')$, which we know to have the desired marginal densities. One can deduce by a symmetric argument that the marginal densities with respect to $Y'$ and $T'$ are also correct.  
\end{proof}


\section{Spectral Estimates}
For this section, we use asymptotic notation, $o()$ and $O()$, to describe the behavior of a function of $n$ as $n\to\infty$. Let $A$ be the adjacency matrix of a random graph $\G=\G(n,p,E,M)$ on $[n]$ satisfying the following: 
\begin{enumerate}
    \item $\Pr[e\in\G]=p\in(0,1)$ for all $e\in{[n]\choose2}$.
    \item There is a fixed constant $E\geq1$ and an $M=M(n)>0$ such that, for every $t\in[0,M]^{[n]\choose2}$, we have $\E\exp\left(\sum_{1\leq i<j\leq n}t_{ij}A_{ij}\right)\leq E\prod_{1\leq i<j\leq n}\left(1-p+pe^{t_{ij}}\right)$.
\end{enumerate}
The choice to make $E$ a fixed constant is just for convenience. All of the results of this section can be easily adapted to work for $E=E(n)$ an arbitrary fixed polynomial in the variable $n$. 

\begin{proposition}\label{gspectrum}
Let $\G$ be defined as above with $p=\frac{\delta\log n}{n}$ ($\delta$ an arbitrary constant), and $M\geq\frac{3}{5}(n/p)^{\frac{1}{2}}$. Then, for any fixed $s>0$, we have with probability at least $1-o(n^{-s})$ that $$v^tAv=O(\sqrt{np})\text{ for all unit vectors }v\perp\one.$$
\end{proposition}
\begin{proof}
This will follow from Lemmas \ref{claim24}, \ref{Light}, \ref{Conditions}, and \ref{mainlemma}. 
\end{proof}

\begin{corollary}\label{redlap}
Let $L:=\Id-D^{-\frac{1}{2}}AD^{-\frac{1}{2}}$ where $D$ is the degree matrix of  $\G=\G\left(n,\frac{\delta\log n}{n},O(1),\frac{3n}{5\sqrt{\delta\log n}}\right)$. Suppose there are positive integers $s=O(1)$ and $m\geq\sqrt{\log n}$ such that  $\Pr[\min_{i\in[n]}\deg_{\G}(i)\geq  m]=1-o(n^{-s})$. Let  $\lambda_1(L)\leq\lambda_2(L)\leq\cdots\leq\lambda_n(L)$ denote the eigenvalues of $L$. Then, with probability $1-o(n^{-s})$, we have $\lambda_1(L)=0$, and 
$$\lambda_2(L)=1-O\left(\frac{\sqrt{\log n}}{m}\right).$$
\end{corollary}
\begin{proof}
Since $\G$ is connected with sufficiently high probability, we will treat $\G$ as though it were connected almost surely. As such, we know that $L$ has minimal eigenvalue $0$ with multiplicity one, and the corresponding eigenvector is $D^{\frac{1}{2}}\one$. Thus we are interested in bounding the quantity
$$\sup\left\{\frac{y^tD^{-\frac{1}{2}}AD^{-\frac{1}{2}}y}{y^ty}:0\neq y\perp D^{\frac{1}{2}}\one\right\}$$
from above by some $\lambda:=O(\sqrt{\log n}/m)$. Equivalently, we would like to show that 
$$x^tAx\leq\lambda x^tDx\text{ for all }x\perp D\one.$$
Without loss of generality, we can assume that $x$ is a unit vector. Let 
$$x=\cos\theta u+\sin\theta v\quad\text{where}\quad u=\frac{1}{\sqrt{n}}\one\text{, }v\perp\one\text{, and }|v|=1.$$
Noting that $u^tAx=u^tDx=0$ and $\cos\theta\frac{\tr D}{\sqrt{n}}=-\sin\theta v^tD\one$ (both of these follow from the assumption that $x\perp D\one$), we have 
$$x^tAx=\sin^2\theta v^tAv-\cos^2\theta\frac{\tr D}{n}\quad\text{and}\quad x^tDx=\sin^2\theta v^tDv-\cos^2\theta\frac{\tr D}{n}.$$
So we have 
\begin{align*}
x^t(\lambda D-A)x&=v^t(\lambda D-A)v\sin^2\theta+\frac{(1-\lambda)\tr D}{n}\cos^2\theta \\ 
&=v^t(\lambda D-A)v\frac{(\tr D)^2}{(\tr D)^2+n(v^tD\one)^2}+\frac{(1-\lambda)\tr D}{n}\frac{n(v^tD\one)^2}{(\tr D)^2+n(v^tD\one)^2}
\end{align*}
which we would like to show is positive. Solving for $v^tAv$, 
it suffices to show that 
$$v^tAv\leq(1-\lambda)\frac{(v^tD\one)^2}{\tr D}+\lambda v^tDv\quad\text{for all unit vectors }v\perp\one.$$
By our minimum degree assumption, it would even suffice to show that $v^tAv\leq\lambda m$. The result now follows from Proposition \ref{gspectrum}. 
\end{proof}

In order to prove Lemma \ref{Light}, we are going to need a version of Bernstein's inequality that works on weighted sums of centered edge indicators of $\G$. 
\begin{theorem}[Bernstein's inequality]
Let $\G$ be as above. Suppose $|c_{ij}|\leq c$ for some fixed $c$ and all $i<j$, and that $\varepsilon\geq0$ is such that $M\geq\frac{\varepsilon}{p\sum_{i<j}c_{ij}^2+2c\varepsilon/3}$. Then, for any $H\subseteq{[n]\choose2}$, we have
\begin{align*}
\Pr\left[\sum_{(i,j)\in H}c_{ij}(A_{ij}-p)\geq\varepsilon\right]&\leq E\exp\left(-\frac{\varepsilon^2}{2p\sum_{(i,j)\in H}c_{ij}^2+2c\varepsilon/3}\right). 
\end{align*} 
\end{theorem}
\begin{proof}
By a Chernoff bound, the numerical bounds $1+x\leq e^x\leq1+x+\frac{3x^2}{6-2x}$ for $x\leq3$, and our assumptions about $\G$, we have for $t\leq\frac{3}{c}$ that 
\begin{align*}
\Pr\left[\sum_{(i,j)\in H}c_{ij}(A_{ij}-p)\geq \varepsilon\right]
&\leq E\inf_{t\in(0,M]}e^{-t\varepsilon}\prod_{(i,j)\in H}\left(1+\E\frac{3t^2c_{ij}^2(A_{ij}-p)^2}{6-2ct}\right) \\ 
&\leq E\inf_{t\in(0,M]}\exp\left(-t\varepsilon+\sum_{(i,j)\in H}\E\frac{3t^2c_{ij}^2(A_{ij}-p)^2}{6-2ct}\right) \\ 
&\leq E\inf_{t\in(0,M]}\exp\left(\frac{(p\sum_{(i,j)\in H}c_{ij}^2+2c\varepsilon/3)t^2-2\varepsilon t}{2-2ct/3}\right) \\ 
&=E\inf_{t\in(0,M]}\exp\left(\frac{(\sigma^2+2c\varepsilon/3)t^2-2\varepsilon t}{2-2ct/3}\right), 
\end{align*}
where $\sigma^2:=p\sum_{(i,j)\in H}c_{ij}^2$. 
Taking 
$t=\frac{\varepsilon}{\sigma^2+2c\varepsilon/3}$, we have
$2-2ct/3=\frac{2\sigma^2+2c\varepsilon/3}{\sigma^2+2c\varepsilon/3}$, and thus
\begin{align*}
\frac{(\sigma^2+2c\varepsilon/3)t^2-2\varepsilon t}{2-2ct/3}=((\sigma^2+2c\varepsilon/3)t-2\varepsilon)\frac{\varepsilon}{\sigma^2+2c\varepsilon/3}\frac{\sigma^2+2c\varepsilon/3}{2\sigma^2+2c\varepsilon/3}=-\frac{\varepsilon^2}{2\sigma^2+2c\varepsilon/3}. 
\end{align*}
Recalling that we were required to assume that $t\leq\frac{3}{c}$, what we have shown holds for our choice of $t$ as long as $\varepsilon\geq-3\sigma^2c^{-1}$.
\end{proof}

We will also want to be able to apply a near--optimal Chernoff bound to uniformly weighted sums of edge indicators. 
\begin{lemma}\label{Chernoff}
For $\G$ as above with $M\geq\log\frac{(1-p)\varepsilon}{1-p\varepsilon}$ where $\varepsilon\geq3$, we have 
$$\Pr\left[\sum_{(i,j)\in H}A_{ij}\geq\varepsilon p|H|\right]\leq E\exp\left(-\frac{\varepsilon\log\varepsilon}{3}p|H|\right).$$
\end{lemma}
\begin{proof}
As with the previous proof, 
\begin{align*}
\Pr\left[\sum_{(i,j)\in H}A_{ij}\geq \varepsilon p|H|\right]&\leq E\inf_{t\in(0,M]}e^{-t\varepsilon p|H|}(1-p+pe^t)^{|H|} \\ 
&=E\left(\frac{(1-p)\varepsilon}{1-p\varepsilon}\right)^{-\varepsilon p|H|}\left(\frac{1-p}{1-p\varepsilon}\right)^{|H|} \\ 
&=E\varepsilon^{-\varepsilon p|H|}\left(1+\frac{(\varepsilon-1)p}{1-p\varepsilon}\right)^{(1-p\varepsilon)|H|} \\ 
&\leq E\exp\left(-(\varepsilon\log\varepsilon-\varepsilon+1)p|H|\right).
\end{align*} 
For $\varepsilon\geq3$, this is bounded above by $E\exp\left(-\frac{\varepsilon\log\varepsilon}{3} p|H|\right)$. 
\end{proof}

To prove Proposition \ref{gspectrum}, we will adapt the Kahn--Szemer\'erdi argument. Let 
$$S:=\left\{v\in\R^n:|v|=1\text{ and }v\perp\one\right\}\quad\text{and}\quad T:=\left\{x\in\frac{1}{\sqrt{4n}}\Z^n:|x|\leq1\text{ and }x\perp\one\right\}.$$
The following lemma is a special case of Claim 2.4 in \cite{FO}. 
\begin{lemma}\label{claim24}
Suppose $|x^tAy|\leq c$ for all $x,y\in T$. Then $|v^tAv|\leq4c$ for all $v\in S$. 
\end{lemma}
We now write 
$$\sum_{(i,j)\in[n]^2}|x_iA_{ij}y_i|=\sum_{(i,j)\in\mathcal{L}}|x_iA_{ij}y_j|+\sum_{(i,j)\in\mathcal{H}}|x_iA_{ij}y_j|$$
where $\mathcal{L}:=\left\{(i,j)\in[n]^2:(x_iy_j)^2\leq\frac{p}{n}\right\}$ and $\mathcal{H}:=[n]^2\setminus\mathcal{L}$. 
\subsection{Light Couples}

\begin{lemma}\label{Light}
Suppose $M\geq\frac{3}{5}(n/p)^{\frac{1}{2}}$ in the definition of $\G$. For any constant $s>0$, we have 
$$\Pr\left[\sum_{(i,j)\in\mathcal{L}}|x_iA_{ij}y_j|\geq 7\sqrt{np}\text{ for some }x,y\in T\right]\leq E(18e^{-3})^n=o(n^{-s}).$$
\end{lemma}

\begin{proof}
We can bound the contribution from the light couples as follows: It is known (\cite{FO}, Claim 2.9) that $|T|\leq18^{n}$. So, by applying a union bound, we have 
$$\Pr\left[\sum_{(i,j)\in\mathcal{L}}|x_iA_{ij}y_j|\geq7\sqrt{np}\text{ for some }x,y\in T\right]\leq18^n\sup_{x,y\in T}\Pr\left[\sum_{(i,j)\in\mathcal{L}}|x_iA_{ij}y_j|\geq7\sqrt{np}\right].$$
Towards applying Bernstein's inequality, define centered random variables 
$$B_{ij}:=\left(|x_iy_j|\one\{(i,j)\in\mathcal{L}\}+|x_jy_i|\one\{(j,i)\in\mathcal{L}\}\right)(A_{ij}-p)$$ 
so that 
$$\sum_{\{i,j\}:(i,j)\in\mathcal{L}\text{ or }(j,i)\in\mathcal{L}}B_{ij}=\sum_{(i,j)\in\mathcal{L}}|x_iA_{ij}y_j|-\E\sum_{(i,j)\in\mathcal{L}}|x_iA_{ij}y_j|.$$
Note that each $B_{ij}^2\leq\frac{4p}{n}$ a.s., and $\sum_{i<j}\E B_{ij}^2\leq p\sum_{i<j}2\big((x_iy_j)^2+(x_jy_i)^2\big)\leq2p$ by taking advantage of the fact that $|x|,|y|\leq1$. Thus, by Bernstein's inequality, \begin{align*}
\Pr\left[\sum_{{\{i,j\}:(i,j)\in\mathcal{L}\text{ or }(j,i)\in\mathcal{L}}}B_{ij}\geq6\varepsilon\sqrt{np}\right]\leq E\exp\left(-\frac{(6\varepsilon)^2np}{4p+6\frac{4\sqrt{p}}{3\sqrt{n}}\sqrt{np}\varepsilon}\right)=E\exp\left(\frac{-9\varepsilon^2}{1+2\varepsilon}n\right). 
\end{align*}
Taking $\varepsilon=1$, this shows that $\sum_{(i,j)\in\mathcal{L}}|x_iA_{ij}y_j|<6\sqrt{np}+\E\sum_{(i,j)\in\mathcal{L}}|x_iA_{ij}y_j|$ with probability at least $1-Ee^{-3n}$. By Lemma 2.6 of \cite{FO}, we also have $\E\sum_{(i,j)\in\mathcal{L}}|x_iA_{ij}y_j|\leq\sqrt{np}$, thus giving us the desired bound. Our use of Bernstein's inequality  requires us to have $M\geq\frac{6\sqrt{np}}{2p+\frac{8\sqrt{p}}{\sqrt{n}}\sqrt{np}}=\frac{3}{5}(n/p)^{\frac{1}{2}}$. \end{proof}

\subsection{Heavy Couples}
For $B,C\subseteq[n]$, let 
$e(B,C):=|\{(i,j)\in\G:i\in B,\,j\in C\}|$ and $\mu(B,C):=p|B||C|$. 
The following is a weakened form of Lemma 9.1 in \cite{HKP}. 
\begin{lemma}\label{Conditions}
Suppose we have constants $c_0,c_1,c_2>1$ and a graph $G$ on $[n]$ with $\max_{i\in[n]}\deg_G(i)\leq c_0np$, and, for all $B,C\subseteq[n]$, one or more of the following hold:
\begin{itemize}
    \item $e(B,C)\leq c_1\mu(B,C)$  
    \item $e(B,C)\log\frac{e(B,C)}{\mu(B,C)}\leq c_2(|B|\vee|C|)\log\frac{n}{|B|\vee|C|}$ 
\end{itemize}
Then $\sum_{(i,j)\in\mathcal{H}}|x_iA_{ij}y_j|= O(\sqrt{np})$. 
\end{lemma}

\begin{lemma}\label{mainlemma}
Suppose that $p=\frac{\delta\log n}{n}$ for some fixed but arbitrary $\delta>0$, and $M\geq\frac{3}{5}(n/p)^{\frac{1}{2}}$. For any $s>0$, there are fixed constants $c_0,c_1,c_2>1$ so that the conditions of Lemma \ref{Conditions} hold for $\G$ with probability at least $1-o(n^{-s})$.  
\end{lemma}
In terms of probabilistic bounds, the proof of this will only rely on Lemma \ref{Chernoff}. We note then that the condition $M\geq\frac{3}{5}(n/p)^{\frac{1}{2}}$ is overkill since Lemma \ref{Chernoff} only ever requires that we have $M$ be greater than a fixed constant.  
\begin{proof}
Assume without loss of generality that $c_0\geq3$. 
First note that, due to the monotonicity of $x\log x$ for $x\geq1$, the second bulleted condition is equivalent to the statement $e(B,C)\leq r_1\mu(B,C)$ where $r_1\geq1$ solves $r_1\log r_1=\frac{c_2(|B|\vee|C|)\log\frac{n}{|B|\vee|C|}}{\mu(B,C)}$. So we can rewrite the two bulleted conditions as the single condition $e(B,C)\leq r\mu(B,C)$ where $r:=r_1\vee c_1$. 
By Lemma \ref{Chernoff} and a union bound, we have 
$$\Pr[\max_{i\in[n]}\deg_{\G}(i)>c_0np]\leq n\Pr[\deg_{\G}(n)\geq c_0np]\leq En\exp\left(-\frac{\delta c_0\log c_0}{3}\log n\right)$$
and 
$$\Pr[e(B,C)>r\mu(B,C)]\leq\Pr[e(B,C)>r_1\mu(B,C)]\leq E\exp\left(-\frac{c_2(|B|\vee|C|)\log\frac{n}{|B|\vee|C|}}{3}\right).$$
Thus we can choose a constant $c_0$ large enough to make 
$\Pr[\max_{i\in[n]}\deg_{\G}(i)>c_0np]=o(n^{-s})$. 
Using this fact and the resulting (high probability) inequality 
$$\frac{e(B,C)}{\mu(B,C)}\leq\frac{c_0(|B|\wedge|C|)np}{p|B||C|}=\frac{c_0n}{|B|\vee|C|},$$ 
we have, in the case $|B|\vee|C|\geq n/e$, that 
$$\Pr[\exists B,C\text{ s.t. }e(B,C)>ec_0\mu(B,C)\text{ and }|B|\vee|C|\geq n/e]\leq\Pr[\max_{i\in[n]}\deg_{\G}(i)>c_0np]=o(n^{-s}).$$
So suppose that $|B|\vee|C|<n/e$. By a union bound over all possible pairs $\{i,j\}\subset\left[\lfloor n/e\rfloor\right]$ ($i\leq j$, without loss of generality) of sizes for the sets $B,C$, it suffices to show that
\begin{align*}
&{n\choose i}{n\choose j} \exp\left(-\frac{c_2j\log\frac{n}{j}}{3}\right)\leq n^{-s-3}.  
\end{align*}
Recalling that ${n\choose j}\leq\left(\frac{ne}{j}\right)^j$ for all $j\in[n]$, this can be done by showing that  
$$(s+3)\log n+j\left(1+\log\frac{n}{j}\right)+i\left(1+\log\frac{n}{i}\right)\leq\frac{c_2}{3}j\log\frac{n}{j}$$
for all $1\leq i\leq j<n/e$. Indeed, since $j\log\frac{n}{j}$ is monotone increasing for $1\leq j<n/e$ (Lemma 2.12 of \cite{FO}), we have 
\begin{align*}
(s+3)\log n+j\left(1+\log\frac{n}{j}\right)+i\left(1+\log\frac{n}{i}\right)\leq(s+3)\log n+4j\log\frac{n}{j}\leq(s+7)j\log\frac{n}{j}. 
\end{align*}
Taking $c_2\geq3s+21$ therefore gives the desired bound. 
\end{proof}

\subsection{Link unions}
Let $T_1,T_2,\dots,T_m$ be jointly independent copies of $\mathcal{T}_{n,k-j}$ with $k$ fixed. Then, if $T_1',T_2'\dots,T_m'$ are jointly independent copies of $\mathcal{T}_{n+j,k}$, we can use Theorem \ref{Link} to couple these random trees so that 
\begin{align*}
L_m^{k,j}(n):&=\operatorname{Link}\left(\{n+1,\dots,n+j\},\bigcup_{i\in[m]}T_i'\right) \\ &=\bigcup_{i\in[m]}\operatorname{Link}\left(\{n+1,\dots,n+j\},T_i'\right) \\ 
&=\mathcal{Y}_{k-j}(n,p)\cup\bigcup_{i\in[m]}T_i,
\end{align*}
where $p:=1-\left(1-\frac{j}{n+j}\right)^m\sim\frac{jm}{n}$. Since each $T_i$ is determinantal, the set of $k$--faces in each $T_i$ are negatively associated (to be abbreviated NA). Moreover, the set of $k$--faces in $T:=\bigcup_{i\in[m]}T_i$ are also NA, as 
$$\one\{f\in T\}=1-\prod_{i\in[m]}(1-\one\{f\in T_i\})$$
is increasing in $T$ as a function $\mathscr{C}_{n,k}\to\R$, and any set of increasing functions defined on disjoint subsets of an NA set is itself an NA set \cite{JP}. By the same reasoning and the fact that sets of jointly independent random variables are NA sets, we also have that the set of $k-j$--faces of $L_m^{k,j}(n)$ is NA. We now narrow our focus to the case $j=k-1$, in which $$G:=L_m^{k,k-1}(n)$$ is a graph with negatively associated edges each appearing with probability $q:=1-(1-p)(1-\frac{2}{n})^m\sim\frac{(k+1)m}{n}$. This implies that $G$ is of type $\G\left(n,q,1,\infty\right)$. Toward applying Corollary \ref{redlap} to this $G$, we will assume that $m=\lceil\delta\log n\rceil$ for some arbitrary constant $\delta>0$. In order to get a sufficiently strong lower bound on the minimum degree of $G$, we need to have a very strong bound on the moment generating function of $\deg_G(n)$ for negative inputs. 
\begin{lemma}
Let $p_0:=1-\left(1-\frac{k-1}{n+k-1}\right)^m\left(1-\frac{1}{n}\right)^m\sim\frac{km}{n}$. Then for all $t<0$ we have  
\begin{align*}
\E\left[\exp\left(t\deg_G(n)\right)\right]&\leq e^{mt}\left(1+p_0(e^t-1)\right)^{n-1-m}\exp\left(\frac{m^2e^{1-t}}{2n}\right). 
\end{align*}
\end{lemma}

\begin{proof}
Let $\ell:=\operatorname{Link}(n,\cup_{i\in[m]}T_i)$, and let $B_i:=\one\{(i,n)\in\G(n,p)\}$, where $\G(n,p)$ is as above and, in particular, independent of $\ell$. Then 
\begin{align*}
\E\left[\exp\left(t\deg_G(n)\right)\big|\ell\right]
&=\prod_{i\in[n-1]}\E\left[\exp\left(t\left(\one\{i\in\ell\}+\one\{i\notin\ell\}B_i\right)\right)\big|\ell\right] \\ 
&=\prod_{i\in[n-1]}\left((1-p)\exp\left(t\one\{i\in\ell\}\right)+pe^t\right).
\end{align*}
Note that $\ell\lawequals\Binomial\left([n-1],p'\right)\cup R$
where $p':=1-(1-\frac{1}{n})^m$, and $R$ is a random subset of $[n-1]$ formed by independently and uniformly choosing (with replacement) $m$ vertices from $[n-1]$---this follows from the $k=1$ case of Theorem \ref{main}. Thus 
\begin{align*}
\E\left[\exp\left(t\deg_G(n)\right)\big|R\right] 
&=\prod_{i\in[n-1]}\E\left[(1-p)\exp\left(t\one\{i\in\ell\}\right)+pe^t\big|R\right]. 
\end{align*}
Now note that 
\begin{align*}
\E\left[\exp\left(t\one\{i\in\ell\}\right)|R\right]&=\E\left[\exp\left(t(\one\{i\in R\}+\one\{i\notin R\}\Bernoulli(p'))\right)|R\right] \\ 
&=\exp\left(t\one\{i\in R\}\right)\left(1-p'+p'\exp\left(t\one\{i\notin R\}\right)\right) \\ 
&=(1-p')\exp\left(t\one\{i\in R\}\right)+p'e^t.
\end{align*}
Thus 
\begin{align*}
\E\left[\exp\left(t\deg_G(n)\right)\big|R\right]&=\prod_{i\in[n-1]}\left((1-p)\left((1-p')\exp\left(t\one\{i\in R\}\right)+p'e^t\right)+pe^t\right) \\ 
&=\prod_{i\in[n-1]}\left((1-p_0)\exp\left(t\one\{i\in R\}\right)+p_0e^t\right). 
\end{align*}
We therefore have 
\begin{align*}
\E\left[\exp\left(t\deg_G(n)\right)\right]&=\sum_{j\in[m]}e^{jt}\left(1-p_0+p_0e^t\right)^{n-1-j}\Pr[|R|=j] \\ 
&=e^{mt}\left(1-p_0+p_0e^t\right)^{n-1-m}\sum_{j\in[m]}e^{(j-m)t}\left(1-p_0+p_0e^t\right)^{m-j}\Pr[|R|=j].
\end{align*}
Since $t<0$, we have $1-p_0+p_0e^t\leq1$, and thus 
\begin{align*}
\E\left[\exp\left(t\deg_G(n)\right)\right]\leq e^{mt}\left(1-p_0+p_0e^t\right)^{n-1-m}\sum_{j\in[m]}e^{-t(m-j)}\Pr[|R|=j].
\end{align*}

Toward controlling the size of $R$, let $R=R_m$, and let $R_j$ be defined as $R$ after only the first $j$ independent samplings of $[n-1]$. 
Then, letting $v_1,v_2,...,v_m$ be the random vertex selections, we have 
$$|R_{j+1}|=|R_j|+\one\{v_{j+1}\notin R_j\}.$$
Thus, for $s>0$, 
\begin{align*}
\E\left[\exp\left(-s|R_{j+1}|\right)\big|R_j\right]&=\exp\left(-s|R_j|\right)\left(\frac{|R_j|}{n-1}+\frac{n-1-|R_j|}{n-1}e^{-s}\right) \\ &\leq\exp\left(-s|R_j|\right)\left(\frac{j}{n-1}+\frac{n-1-j}{n-1}e^{-s}\right)=\exp\left(-s|R_j|-s\right)\left(1+\frac{j(e^s-1)}{n-1}\right).
\end{align*}
Taking expectations and iterating then gives 
$$\E\exp\left(-s|R_m|\right)\leq e^{-ms}\prod_{i=1}^{m-1}\left(1+\frac{i(e^s-1)}{n-1}\right).$$ 
Thus, by Chernoff's inequality, we have 
$$\Pr[|R_m|\leq m-j]\leq e^{-js}\prod_{i=1}^{m-1}\left(1+\frac{i(e^s-1)}{n-1}\right)\leq\exp\left(-js+{m\choose 2}\frac{e^s-1}{n-1}\right)$$
for any $s>0$. Taking $s=\log\frac{2j(n-1)}{m(m-1)}$, gives 
$$\Pr[|R_m|\leq m-j]\leq\exp\left(j-\frac{m(m-1)}{2(n-1)}\right)\left(\frac{m(m-1)}{2j(n-1)}\right)^j\leq\left(\frac{em^2}{2jn}\right)^j.$$
Thus $|R|\geq m-j+1$ with probability at least  $1-\left(\frac{em^2}{2jn}\right)^j$. Recall the summation by parts formula: 
$$\sum_{j\in[m]}f_jg_j=f_{m}\sum_{j\in[m]}g_j-\sum_{j\in[m]}(f_j-f_{j-1})\sum_{i\in[j-1]}g_i.$$
Setting $f_j=e^{-t(m-j)}$ and $g_j=\Pr[|R|=j]$, we have 
\begin{align*}
\sum_{j\in[m]}e^{-t(m-j)}\Pr[|R|=j]&=1+\left(e^{-t}-1\right)\sum_{j\in[m]}\Pr[|R|<j]e^{-t(m-j)} \\ 
&\leq1+\left(e^{-t}-1\right)\sum_{j\in[m]}\left(\frac{em^2}{2(m-j+1)n}\right)^{m-j+1}e^{-t(m-j)} \\ 
&=1+\left(1-e^t\right)\sum_{j\in[m]}\left(\frac{e^{1-t}m^2}{2jn}\right)^j \\ 
&\leq1+\left(1-e^t\right)\left(\exp\left(\frac{m^2e^{1-t}}{2n}\right)-1\right), 
\end{align*}
where the last inequality uses the bound $j^j\geq j!$. 
Thus, for $t<0$ we have 
\begin{align*}
\E\left[\exp\left(t\deg_G(n)\right)\right]&\leq e^{mt}\left(1-p_0+p_0e^t\right)^{n-1-m}\sum_{j\in[m]}e^{-t(m-j)}\Pr[|R|=j] \\ 
&\leq e^{mt}\left(1+p_0(e^t-1)\right)^{n-1-m}\left(1+(1-e^t)\left(\exp\left(\frac{m^2e^{1-t}}{2n}\right)-1\right)\right) \\ 
&\leq e^{mt}\left(1+p_0(e^t-1)\right)^{n-1-m}\exp\left(\frac{m^2e^{1-t}}{2n}\right). 
\end{align*}
\end{proof}

\begin{lemma}
For $0\leq j\leq m$, we have $\Pr[\min_{i\in[n]}\deg_G(i)\leq m-j]\leq(1+o(1))n^{1-j}e^{-km}$.
\end{lemma}

\begin{proof}
For any $t>0$, we have by a union bound that 
\begin{align*}
\Pr[\min_{i\in[n]}\deg_G(i)\leq m-j]&\leq n\Pr[\exp\left(-t\deg_G(n)\right)\geq e^{(j-m)t}] \\ 
&\leq ne^{-jt}\left(1+p_0(e^{-t}-1)\right)^{n-1-m}\exp\left(\frac{m^2e^{1+t}}{2n}\right) \\ 
&\leq\exp\left(\log n-jt+(n-1-m)p_0(e^{-t}-1)+\frac{m^2e^{1+t}}{2n}\right). 
\end{align*}
Letting $t=r\log n$ for any $r\in(0,1)$ gives 
\begin{align*}
\Pr[\min_{i\in[n]}\deg_G(i)\leq m-j]&\leq\exp\left((1-jr)\log n-(n-1-m)p_0(1-n^{-r})+\frac{em^2}{2n^{1-r}}\right) \\ &\leq(1+o(1))\exp\left((1-jr)\log n-km\right).
\end{align*}
Taking $r\to1$ from the left then gives $(1+o(1))n^{1-j}e^{-km}$. 
\end{proof}

\begin{lemma}\label{forapp}
Let $L$ be the reduced Laplacian of $G$ as defined above with $m=\lceil\delta\log n\rceil$ and $\delta>0$ an arbitrary constant (so that $\Pr[\min_{i\in[n]}\deg_G(i)\geq m-j+2]=1-o(n^{-j})$ for any fixed $j$). Then for any fixed $s$, we have $\lambda_1(L)=0$ and 
$\lambda_2(L)=1-O\left(1/\sqrt{\log n}\right)$
with probability $1-o(n^{-s})$. 
\end{lemma}

\begin{proof}
This follows from the previous lemma and Corollary \ref{redlap}. 
\end{proof}

\newpage 

\printbibliography[heading=bibliography]

@article{HKP,
    author = {Hoffman, Christopher and Kahle, Matthew and Paquette, Elliot},
    title = "{Spectral Gaps of Random Graphs and Applications}",
    journal = {International Mathematics Research Notices},
    volume = {2021},
    number = {11},
    pages = {8353-8404},
    year = {2019},
    month = {05},
    issn = {1073-7928},
    doi = {10.1093/imrn/rnz077},
    url = {https://doi.org/10.1093/imrn/rnz077},
    eprint = {https://academic.oup.com/imrn/article-pdf/2021/11/8353/42309307/rnz077.pdf},
}

@article{FO,
author = {Feige, Uriel and Ofek, Eran},
title = {Spectral techniques applied to sparse random graphs},
journal = {Random Structures \& Algorithms},
volume = {27},
number = {2},
pages = {251-275},
doi = {https://doi.org/10.1002/rsa.20089},
url = {https://onlinelibrary.wiley.com/doi/abs/10.1002/rsa.20089},
eprint = {https://onlinelibrary.wiley.com/doi/pdf/10.1002/rsa.20089},
year = {2005}
}

@article{Kalai,
  title={Enumeration of Q-acyclic simplicial complexes},
  author={Gil Kalai},
  journal={Israel Journal of Mathematics},
  year={1983},
  volume={45},
  pages={337-351}
}

@article{Lyons,
author = {Lyons, Russell},
journal = {Publications Mathématiques de l'IHÉS},
language = {eng},
pages = {167-212},
publisher = {Springer},
title = {Determinantal probability measures},
url = {http://eudml.org/doc/104195},
volume = {98},
year = {2003},
}

@unknown{Meszaros,
author = {Mészáros, András},
year = {2021},
month = {01},
pages = {},
title = {The local weak limit of $k$-dimensional hypertrees},
doi = {10.1090/tran/8711}
}

@inproceedings{Lub,
  title={High dimensional expanders},
  author={Lubotzky, Alexander},
  booktitle={Proceedings of the International Congress of Mathematicians: Rio de Janeiro 2018},
  pages={705--730},
  year={2018},
  organization={World Scientific}
}

@article{Eig,
author = {Gundert, Anna and Wagner, Uli},
year = {2014},
month = {11},
pages = {},
title = {On Eigenvalues of Random Complexes},
volume = {216},
journal = {Israel Journal of Mathematics},
doi = {10.1007/s11856-016-1419-1}
}

@article{Opp2,
  title={Local spectral expansion approach to high dimensional expanders part II: Mixing and geometrical overlapping},
  author={Oppenheim, Izhar},
  journal={Discrete \& Computational Geometry},
  volume={64},
  number={3},
  pages={1023--1066},
  year={2020},
  publisher={Springer}
}

@article{DKM1,
title = {Cellular spanning trees and Laplacians of cubical complexes},
journal = {Advances in Applied Mathematics},
volume = {46},
number = {1},
pages = {247-274},
year = {2011},
note = {Special issue in honor of Dennis Stanton},
issn = {0196-8858},
doi = {https://doi.org/10.1016/j.aam.2010.05.005},
url = {https://www.sciencedirect.com/science/article/pii/S0196885810001168},
author = {Art M. Duval and Caroline J. Klivans and Jeremy L. Martin}
}

@article{DKM2,
author = {Duval, Art M. and Klivans, Caroline J. and Martin, Jeremy L.},
title = {Cuts and Flows of Cell Complexes},
year = {2015},
issue_date = {June      2015},
publisher = {Kluwer Academic Publishers},
address = {USA},
volume = {41},
number = {4},
issn = {0925-9899},
url = {https://doi.org/10.1007/s10801-014-0561-2},
doi = {10.1007/s10801-014-0561-2},
journal = {J. Algebraic Comb.},
month = {jun},
pages = {969–999},
numpages = {31}
}

@article{BK,
author = {Bernardi, Olivier and Klivans, Caroline},
year = {2015},
month = {12},
pages = {},
title = {Directed Rooted Forests in Higher Dimension},
volume = {23},
journal = {The Electronic Journal of Combinatorics},
doi = {10.37236/5819}
}

@article{JP,
author = {Kumar Joag-Dev and Frank Proschan},
title = {{Negative Association of Random Variables with Applications}},
volume = {11},
journal = {The Annals of Statistics},
number = {1},
publisher = {Institute of Mathematical Statistics},
pages = {286 -- 295},
year = {1983},
doi = {10.1214/aos/1176346079},
URL = {https://doi.org/10.1214/aos/1176346079}
}

@article{HomCon,
  title={Homological Connectivity Of Random 2-Complexes},
  author={Nathan Linial and Roy Meshulam},
  journal={Combinatorica},
  year={2006},
  volume={26},
  pages={475-487}
}

@article{MeshWall,
  title={Homological connectivity of random k‐dimensional complexes},
  author={Roy Meshulam and N. Wallach},
  journal={Random Structures \& Algorithms},
  year={2009},
  volume={34}
}

@article{coja, 
  title={The spectral gap of random graphs with given expected degrees},
  author={Coja-Oghlan, Amin and Lanka, Andr{\'e}},
  journal={the electronic journal of combinatorics},
  pages={R138--R138},
  year={2009}
}

@article{Zuk1,
  title={La propri{\'e}t{\'e} (T) de Kazhdan pour les groupes agissant sur les polyedres},
  author={\.Zuk, Andrzej},
  journal={Comptes rendus de l'Acad{\'e}mie des sciences. S{\'e}rie 1, Math{\'e}matique},
  volume={323},
  number={5},
  pages={453--458},
  year={1996}
}

@article{Zuk2,
author = {\.Zuk, Andrzej},
year = {2003},
month = {06},
pages = {643-670},
title = {Property (T) and Kazhdan constants for discrete groups},
volume = {13},
journal = {Geometric And Functional Analysis},
doi = {10.1007/s00039-003-0425-8}
}

@misc{KN,
      title={Topology and geometry of random 2-dimensional hypertrees}, 
      author={Matthew Kahle and Andrew Newman},
      year={2020},
      eprint={2004.13572},
      archivePrefix={arXiv},
      primaryClass={math.AT}
}

@article{LP,
author = {Linial, Nati and Peled, Yuval},
title = {Enumeration and randomized constructions of hypertrees},
journal = {Random Structures \& Algorithms},
volume = {55},
number = {3},
pages = {677-695},
doi = {https://doi.org/10.1002/rsa.20841},
url = {https://onlinelibrary.wiley.com/doi/abs/10.1002/rsa.20841},
eprint = {https://onlinelibrary.wiley.com/doi/pdf/10.1002/rsa.20841},
year = {2019}
}

@article{STT,
author = {Duval, Art and Klivans, Caroline and Martin, Jeremy},
year = {2008},
month = {02},
pages = {},
title = {Simplicial Matrix-Tree Theorems},
volume = {361},
journal = {Transactions of the American Mathematical Society},
doi = {10.1090/S0002-9947-09-04898-3}
}

@article{Opp1,
  title={Local spectral expansion approach to high dimensional expanders part I: Descent of spectral gaps},
  author={Oppenheim, Izhar},
  journal={Discrete \& Computational Geometry},
  volume={59},
  number={2},
  pages={293--330},
  year={2018},
  publisher={Springer}
}

@article{Garland,
  title={p-Adic Curvature and the Cohomology of Discrete Subgroups of p-Adic Groups},
  author={Howard Garland},
  journal={Annals of Mathematics},
  year={1973},
  volume={97},
  pages={375}
}

@article{oliveira,
  title={Concentration of the adjacency matrix and of the Laplacian in random graphs with independent edges},
  author={Oliveira, Roberto Imbuzeiro},
  journal={arXiv preprint arXiv:0911.0600},
  year={2009},
  publisher={Citeseer}
}

@article{dense,
author = {Konstantin Tikhomirov and Pierre Youssef},
title = {{The spectral gap of dense random regular graphs}},
volume = {47},
journal = {The Annals of Probability},
number = {1},
publisher = {Institute of Mathematical Statistics},
pages = {362 -- 419},
year = {2019},
doi = {10.1214/18-AOP1263},
URL = {https://doi.org/10.1214/18-AOP1263}
}

@article{Lyons2,
author = {Lyons, Russell},
title = {Random complexes and $\ell^2$-Betti numbers},
journal = {Journal of Topology and Analysis},
volume = {01},
number = {02},
pages = {153-175},
year = {2009},
doi = {10.1142/S1793525309000072},

URL = { 
        https://doi.org/10.1142/S1793525309000072
    
},
eprint = { 
        https://doi.org/10.1142/S1793525309000072
    
}
}

\end{document}